\documentclass[reqno]{amsart}
\usepackage{microtype}

\usepackage[shortlabels]{enumitem}
\setlist[enumerate]{label={\arabic*.}}
\setlist[description]{font=\normalfont\slshape}

\usepackage{amssymb}
\usepackage{mathtools}

\usepackage[dvipsnames]{xcolor}
\usepackage[backref]{hyperref}
\usepackage[abbrev,shortalphabetic]{amsrefs}
\hypersetup{colorlinks=true,linkcolor=Maroon,citecolor=OliveGreen}

\usepackage{cleveref}
\usepackage[norefs, nocites]{refcheck}

\makeatletter
\newcommand{\refcheckize}[1]{%
  \expandafter\let\csname @@\string#1\endcsname#1%
  \expandafter\DeclareRobustCommand\csname relax\string#1\endcsname[1]{%
    \csname @@\string#1\endcsname{##1}\wrtusdrf{##1}}%
  \expandafter\let\expandafter#1\csname relax\string#1\endcsname
}
\makeatother
\refcheckize{\cref}
\refcheckize{\Cref}
\theoremstyle{plain}
\newtheorem{theorem}{Theorem}
\newtheorem{lemma}[theorem]{Lemma}
\newtheorem{proposition}[theorem]{Proposition}
\newtheorem{corollary}[theorem]{Corollary}

\theoremstyle{definition}

\newtheorem{remark}[theorem]{Remark}

\newtheorem{example}[theorem]{Example}

\numberwithin{theorem}{section}
\numberwithin{equation}{section}

\renewcommand\bar{\overline}

\newcommand{\eps}{\epsilon}
\renewcommand\subset{\subseteq}
\renewcommand\supset{\supseteq}

\renewcommand\phi{\varphi}

\newcommand\floor[1]{\left\lfloor{#1}\right\rfloor}

\newcommand\br[1]{\left(#1\right)}

\def\nsgp{\trianglelefteq}

\newcommand\opr[1]{\operatorname{#1}}

\def\F{\mathbf{F}}

\def\P{\mathbf{P}}

\def\rad{\opr{rad}}  % solvable radical
\def\Comm{\opr{Comm}}

\def\cod{\opr{cod}}

\def\gg{\mathfrak{g}}
\def\hh{\mathfrak{h}}

\def\II{\mathcal{I}}
\def\JJ{\mathcal{J}}

\usepackage{todonotes}

\begin{document}

\title{Probabilistically nilpotent groups of class two}

\author{Sean Eberhard}
\address{Sean Eberhard, Centre for Mathematical Sciences, Wilberforce Road, Cambridge CB3~0WB, UK}
\email{eberhard@maths.cam.ac.uk}

\author{Pavel Shumyatsky}
\address{Pavel Shumyatsky, Department of Mathematics, University of Brasilia, Brasilia-DF, 70910-900 Brazil}
\email{pavel@unb.br}

\thanks{The first author has received funding from the European Research Council (ERC) under the European Union’s Horizon 2020 research and innovation programme (grant agreement No. 803711).
The second author was supported by FAPDF and CNPq-Brazil.}

\begin{abstract}
    For $G$ a finite group, let $d_2(G)$ denote the proportion of triples $(x, y, z) \in G^3$ such that $[x, y, z] = 1$.
    We determine the structure of finite groups $G$ such that $d_2(G)$ is bounded away from zero:
    if $d_2(G) \geq \eps > 0$,
    $G$ has a class-4 nilpotent normal subgroup $H$ such that $[G : H] $ and $|\gamma_4(H)|$ are both bounded in terms of $\eps$.
    We also show that if $G$ is an infinite group whose commutators have boundedly many conjugates,
    or indeed if $G$ satisfies a certain more general commutator covering condition,
    then $G$ is finite-by-class-3-nilpotent-by-finite.
\end{abstract}

\maketitle

\setcounter{tocdepth}{1}
\tableofcontents

\section{Introduction}

If $G$ is a finite group, we define the \emph{class-$k$ nilpotency degree} by
\[
    d_k(G) = |\{(x_1, \dots, x_{k+1}) \in G^{k+1} : [x_1, \dots, x_{k+1}] = 1\}| / |G^{k+1}|.
\]
In probablistic notation,
\[
    d_k(G) = \P_{x_1, \dots, x_{k+1} \in G}([x_1, \dots, x_{k+1}] = 1).
\]
Then $G$ is nilpotent of class $k$ if and only if $d_k(G) = 1$,
so class-$k$ nilpotency degree is a statistical relaxation of class-$k$ nilpotency.
We may call $G$ \emph{probabilistically nilpotent of class $k$} if $d_k(G)$ is bounded away from zero
(so $G$ has statistically significant class-$k$ nilpotency).

For example, $d_1(G)$ is the probability that two random elements of $G$ commute,
sometimes called the \emph{commuting probability} of $G$.
It is well-known that $d_1(G) \leq 5/8$ for any nonabelian group $G$.
Another important result, less well known, is a theorem of Peter Neumann
which states that if $d_1(G)$ is bounded away from zero
then $G$ has a subgroup $H$
such that $[G:H]$ and $|H'|$ are both bounded.
Thus a finite probabilistically abelian group is bounded-by-abelian-by-bounded.%
\footnote{
Here and throughout the paper we say ``$(a, b, c, \dots)$-bounded'' to mean ``bounded by a function of $a, b, c, \dots$ only''.
Where the parameters $a, b, c, \dots$ are clear from the context we often just say ``bounded''.
Also, if we say a group is $A$-by-$B$-by-$\cdots$-by-$Z$, we mean there is a sequence of subgroups $1 = G_0 \nsgp G_1 \nsgp \cdots \nsgp G_n = G$
such that $G_1/G_0$ is $A$, $G_2/G_1$ is $B$, ..., $G_n/G_{n-1}$ is $Z$.}

It is natural to ask for an analogous qualitative description of probabilistically nilpotent groups of class $k$.
Essentially this question has been asked by several authors, including
Shalev~\cite{shalev}, Martino, Tointon, Valiunas, and Ventura~\cite{MTVV}*{Question~1.10}, and Green [personal communication].
If $G$ has boundedly many generators then it was shown by Shalev~\cite{shalev}*{Theorem~1.1} that $G$ is class-$k$-by-bounded.
If not, the structure of $G$ is much less transparent,
and certainly any structure theorem must at least include the class of bounded-by-class-$k$-by-bounded groups.

In this paper we consider finite groups $G$ with $d_2(G)$ bounded away from zero.
We show that $G$ need not be bounded-by-class-2-by-bounded (as might be hoped),
nor even class-3-by-bounded,
but we will show that $G$ must be bounded-by-class-3-by-bounded.

\begin{theorem}
    \label{thm:d2-theorem}
    Let $G$ be a finite group such that $d_2(G) \geq \eps > 0$.
    Then $G$ has a subgroup $H$ of nilpotency class at most 4
    such that $[G:H]$ and $|\gamma_4(H)|$ are both $\eps$-bounded.
\end{theorem}

Further details about the structure of $G$ can be extracted from the proof of the theorem. It is noteworthy that if $T$ is the trilinear map $(H / H')^3 \to \gamma_3(H) / \gamma_4(H)$ induced by the triple commutator $[x, y, z]$, then there is an expression for the map $T$ of the form
\[
    T(x, y, z) = A(a(x, y), z) + B(b(x, z), y) + C(c(y, z), x),
\]
where $A, B, C, a, b, c$ are bilinear maps and $a, b, c$ have $\eps$-bounded codomains.
Conversely, if $G$ has this structure then $d_2(G)$ is bounded away from zero by a function of $[G:H]$, $|\gamma_4(H)|$,
and the size of the codomains of $a, b, c$.

Much of the proof of \Cref{thm:d2-theorem} consists of a study of groups satisfying a certain commutator covering condition.
If $X$ and $Y$ are subsets of a group $G$, let $\Comm(X, Y)$ denote the \emph{set} of commutators $\{[x, y] : x \in X, y \in Y\}$.
Let $n \geq1$, let $B = \{x \in G: |x^G| \leq n\}$, and let $S \subset G'$.
Then we say $G$ satisfies the \emph{commutator covering condition} (with given $n$ and $S$) if
\begin{equation}
    \label{eq:covering-condition}
    \Comm(G, G) \subset B S.
\end{equation}
We will prove that a finite group with $d_2(G)$ bounded away from zero has a bounded-index subgroup satisfying the covering condition
with bounded $n$ and $|S|$, and conversely the covering condition implies $d_2(G)$ is bounded away from zero by a function of $n$ and $|S|$,
so for finite groups the condition that $d_2(G)$ is bounded away from zero and the covering condition are loosely equivalent.
On the other hand the covering condition makes sense for infinite as well as finite groups,
and is satisfied by groups with boundedly finite conjugacy classes of commutators (the case $S = \{1\}$), which were studied in \cite{dieshu}.
We will prove the following structure theorem for groups (finite or infinite) satisfying the commutator covering condition.

\begin{theorem}
    \label{thm:covered-groups}
    Let $G$ be a group satisfying the commutator covering condition \eqref{eq:covering-condition}.
    Then $G$ has a subgroup $H$ of nilpotency class at most 4 such that $[G:H]$ and $|\gamma_4(H)|$
    are both finite and $(n, |S|)$-bounded.
\end{theorem}

We do not know whether the techniques used in this paper can be adapted to handle groups $G$ with $d_k(G)$ bounded away from zero for $k \geq 3$ (that is, probabilistically nilpotent groups of class $k$),
or even more specially groups $G$ in which all weight-$k$ commutators $[x_1, \dots, x_k]$ have boundedly many conjugates.
It may be that any such group has a subgroup $H$ for which $[G:H]$ and $|\gamma_{k+1}(H)|$ are bounded.
However, where $X_k = \{[x_1, \dots, x_k] : x_1, \dots, x_k \in G\}$ is the set of weight-$k$ commutators, it is known that
\begin{itemize}
    \item if $|x^G| \leq n$ for all $x \in X_k$ then $|\gamma_k(G)'|$ is finite and $(k, n)$-bounded~\cite{detmorshu};
    \item if $|g^{X_k}| \leq n$ for all $g \in G$ then $|\gamma_{k+1}(G)|$ is finite and $(k, n)$-bounded~\cite{detdonmorshu}.
\end{itemize}

\section{Example}
%\label{sec:example}

In this section we construct a family of finite groups $G$ for which $d_2(G)$ is bounded away from zero but such that the largest class-3 subgroup of $G$ does not have bounded index.
A slight variant also gives an infinite group $G$ in which commutators have boundedly many conjugates but such that $G$ is not virtually class-3.
In both cases $G$ is nilpotent of class 4 and bounded-by-class-3.

Let $K = \F_p$ where $p$ is a small prime (say 2, 3, or 5)
and let $V$ be a vector space over $K$.
Let $f : V \times V \to K$ be a (generic) bilinear map.
We define a graded unital $K$-algebra
\[
    R = R_0 \oplus R_1 \oplus R_2 \oplus R_3 \oplus R_4
\]
with
\begin{align*}
  R_0 &= K,\\
  R_1 &= V,\\
  R_2 &= V \otimes V,\\
  R_3 &= V,\\
  R_4 &= K,
\end{align*}
and $R_k = 0$ for $k > 4$.
We must define the multiplication maps $R_i \times R_j \to R_{i+j}$ for all $i, j \geq 0$.
The unit 1 of $R_0 = K$ is defined to be a unit of $R$.
Multiplication $R_1 \times R_1 \to R_2$ is defined to be the universal map $V \times V \to V \otimes V$.
Let $\theta : R_1 \to R_3$ be an isomorphism.
We define multiplication between $R_1$ and $R_2$ and between $R_1$ and $R_3$ by the rules
\begin{align*}
    xyz &= f(y, z) x^\theta + f(x, y) z^\theta && (x, y, z \in R_1) \\
    x^\theta y &= x y^\theta = f(x, y) && (x, y \in R_1).
\end{align*}
Multiplication $R_2 \times R_2 \to R_4$ is then determined by $(xy)(zw) = x(yzw)$.

Associativity holds on $R_1 \times R_1 \times R_1$ by definition,
so it suffices to check associativity on $R_2 \times R_1 \times R_1$, $R_1 \times R_2 \times R_1$, and $R_1 \times R_1 \times R_2$.
Since $(xy)(zw) = x(yzw)$ for all $x, y, z, w \in R_1$, it suffices to verify that $(xyz)w = x(yzw)$ for $x, y, z, w \in R_1$, and this is a straightforward check:
\begin{align*}
    (xyz)w = f(y, z) (x^\theta w) + f(x, y) (z^\theta w) = f(y, z) f(x, w) + f(x, y) f(z, w), \\
    x(yzw) = f(z, w) (x y^\theta) + f(y, z) (x w^\theta) = f(z, w) f(x, y) + f(y, z) f(x, w).
\end{align*}

Let $f^S(x, y) = f(x, y) + f(y, x)$ and $f^A(x, y) = f(x, y) - f(y, x)$.
Write $[x, y]_L = xy - yx$ for the Lie bracket ($x, y \in R$).
Then we compute (for $x, y, z, w \in R_1$)
\begin{align*}
    [x, y]_L z &= f(y, z) x^\theta - f(x, z) y^\theta + f^A(x, y) z^\theta \\
    z [x, y]_L &= f^A(x, y) z^\theta + f(z, x) y^\theta - f(z, y) x^\theta \\
    [x, y, z]_L &= f^S(y, z) x^\theta - f^S(x, z) y^\theta \\
    [x, y, z, w]_L &= f^A(x, w) f^S(y, z) - f^A(y, w) f^S(x, z).
\end{align*}

Assume $f^S$ and $f^A$ are both nondegenerate.
Let $H \leq R_1$ be a subspace of finite codimension $d$.
If $2d < \dim V$ (or $V$ is infinite-dimensional), we can find $x, w \in H$ such that $f^A(x, w) \neq 0$.
Let $H_1 = H \cap \ker f^S(x, \cdot)$.
Then $H_1$ has codimension at most $d+1$, and if $d + (d+1) < \dim V$ then we can find $y \in H$ and $z \in H_1$
such that $f^S(y, z) \neq 0$. Since $f^S(x, z) = 0$, it follows that
\[
    [x, y, z, w]_L = f^A(x, w) f^S(y, z) \neq 0.
\]
Hence if $[x, y, z, w]_L = 0$ identically on $H$ then $H$ has codimension at least $(\dim V - 1) / 2$.

Let $L_i = \bigoplus_{k \geq i} R_k$ and let $G = 1 + L_1$.
Then $G$ is a nilpotent group of class 4.
The group commutator and Lie bracket are related by
\[
  [1 + x, 1 + y] = 1 + (1+x)^{-1} (1 + y)^{-1} [x, y]_L.
\]
It follows that
\[
    [1 + x, 1 + y, 1 + z] = 1 + [x, y, z]_L \pmod {L_4}.
\]
For any $x, y \in L_1$ the formula for $[x, y, z]_L$ shows that the set of $z \in L_1$ for which $[x, y, z]_L = 0$
is a subspace of codimension at most 2.
Hence the group of $1 + z \in G$ for which $[1+x, 1+y, 1+z] \in 1 + L_4$ has index at most $p^2$.
Since $|1+L_4| = |K| = p$, it follows that $[1+x, 1+y]$ has at most $p^3$ conjugates.
Hence every commutator has boundedly many conjugates.
On the other hand, the quadruple commutator is
\[
    [1+x, 1+y, 1+z, 1+w] = 1 + [x, y, z, w]_L.
\]
By our earlier remarks about $[x, y, z, w]_L$, if $H \leq G$ has class 3 then $[G:H]$ is at least $p^{(\dim V - 1)/2}$.

To give a concrete example let $V$ be the $\F_2$-vector space with basis $(e_i, e'_i : i \in I)$ (for some index set $I$) and let
\begin{align*}
    &f(e_i, e_j) = f(e'_i, e_j) = f(e'_i, e'_j) = 0\\
    &f(e_i, e'_j) = \delta_{ij}.
\end{align*}
Then $f^S = f^A$ is the bilinear map with matrix
\[
    \begin{pmatrix}
        0 & I \\
        I & 0
    \end{pmatrix},
\]
which is clearly nondegenerate.
By taking $I = \{1, \dots, n\}$ we get a family of finite groups with $d_2(G) \geq 1/2^3$ and no bounded-index class-3 subgroup.
By taking $I = \{1, 2, 3, \dots\}$ we get an infinite group $G$ in which every commutator has at most $2^3$ conjugates and yet $G$ has no finite-index class-3 subgroup.

\begin{remark}
    A similar but weaker example (class-3, but not bounded-by-class-2-by-bounded) appeared in \cite{eberhard-thesis}*{Section~2.7}.
\end{remark}

\section{Outline}

In this section we give an outline of the proof and in the next we list some tools that we need.
The proof of the main theorems then occupies the rest of the paper and consists of the following steps.
\begin{enumerate}
    \item In \Cref{sec:metric-neumann} we abstract the method in the proof of Peter Neumann's theorem (\Cref{thm:neumann})
    and we obtain a general statement about groups with an invariant seminorm.
    Applied to the seminorm $\|x\| = \log|x^G|$, the result is that if $d_2(G)$ is bounded away from zero then
    $G$ has a subgroup of bounded index
    satisfying the covering condition \eqref{eq:covering-condition} with $n$ and $|S|$ bounded in terms of $d_2(G)$.
    Thus \Cref{thm:d2-theorem} reduces to \Cref{thm:covered-groups}.
    \item In \Cref{sec:to-bounded-derived-length} we show that a group satisfying the covering condition is virtually soluble of bounded derived length.
    In the special case $S = \{1\}$ (groups in which commutators have boundedly many conjugates),
    it was shown in \cite{dieshu} that $G$ is bounded-by-metabelian.
    For arbitrary $S$ we use induction on $|S|$ and adapt the method of \cite{dieshu}.
    \item In \Cref{sec:to-bounded-class} we show that a soluble group of bounded derived length satisfying the covering condition
    has a nilpotent subgroup of bounded index and class.
    A key tool in this part of the proof is an asymmetric version of Neumann's theorem about $d_1(G)$.
    We also use Hall's criterion for nilpotency, which enables us to reduce to the metabelian case.
    In the metabelian case we argue that all elements in an appropriate subgroup satisfy the Engel identity,
    and it follows that the subgroup is nilpotent of bounded class.
    \item In \Cref{sec:to-class-4} we complete the proof by showing that a bounded-class nilpotent group satisfying the covering condition
    is bounded-by-class-3-by-bounded.
    To do this, we use induction on nilpotency class to reduce to the class-4 case.
    In the class-4 case, we use the theory of \emph{multilinear bias} (or \emph{analytic rank}).
    The quadruple commutator $[x, y, z, w]$ is a uniformly biased 4-linear map,
    and a version of the Jacobi identity harshly restricts the form of this map,
    and we can conclude that $G$ has a subgroup $H$ of bounded index with $|\gamma_4(H)|$ bounded,
    as desired.
\end{enumerate}

\section{Tools}

\subsection{The Neumanns' theorems}

Key tools and prototypes for our main theorems are a pair of results by a pair of Neumanns.
The first is a result due to Bernhard Neumann describing the structure of a \emph{BFC group}, a group in which $|x^G| \leq n$ for all $x \in G$.
The conclusion is that $|G'|$ is finite,
and conversely it is clear that if $G'$ is finite then $|x^G|$ is finite and $|G'|$-bounded for every $x \in G$.

\begin{theorem}
    [B.~Neumann's BFC theorem~\cite{bneumann}; see also \cite{robinson}*{14.5.11}]
    \label{thm:BFC}
    Let $G$ be a group in which $|x^G| \leq n$ for every $x \in G$.
    Then $|G'|$ is finite and $n$-bounded.
\end{theorem}

An explicit bound for $|G'|$ was first established by Wiegold~\cite{wiegold}.
Subsequently several people worked on the bound for $|G'|$.
%Segal and Shalev proved in \cite{segal-shalev} that $|G'| \leq n^{\frac12(13 + \log_2 n)}$.
Guralnick and Mar\'oti proved in \cite{guralnick-maroti}*{Theorem~1.9} that $|G'| \leq n^{(7 + \log_2 n)/2}$.

The second is the theorem of Peter Neumann about $d_1$ already mentioned in the introduction.

\begin{theorem}
    [P.~Neumann~\cite{pneumann}]
    \label{thm:neumann}
    Let $G$ be a finite group such that $d_1(G) \geq \eps > 0$.
    Then $G$ has a normal subgroup $H$ such that $[G:H]$ and $|H'|$ are $\eps$-bounded.
\end{theorem}

This theorem bears roughly the same relation to \Cref{thm:BFC} as \Cref{thm:d2-theorem} bears to \Cref{thm:covered-groups}.

Actually we will use an asymmetric version of Peter Neumann's theorem that considers not just $d_1(G)$
but the commuting probability between two normal subgroups $A, B \nsgp G$.

\begin{theorem}
    [P.~Neumann's theorem, asymmetric version]
    \label{thm:asymmetric-neumann}
    Let $G$ be a finite group with normal subgroups $A, B \nsgp G$
    such that $\P_{a \in A, b \in B}([a, b] = 1) \geq 1/C$.
    Then there are $C$-bounded-index subgroups $H \leq A$ and $K \leq B$,
    both normal in $G$,
    such that $|[H, K]|$ is $C$-bounded.
\end{theorem}

To our knowledge this result does not appear in the literature, though it is similar to (and easier than) \cite{DS}*{Proposition~1.2}.
It can be established by adapting the proof in the symmetric case appropriately.
This result is also a corollary of a more general result we will prove, so a proof will be given in \Cref{sec:metric-neumann}.

\subsection{Tools for reducing to the finite case}

In \Cref{thm:covered-groups} we allow $G$ to be infinite,
mostly because we can reduce to the finite case whenever $G$ is locally residually finite.

Say that $G$ is \emph{$m$-by-class-$s$-by-$n$}
if $G$ has a subgroup $H$ such that $[G : H] \leq n$ and $|\gamma_s(H)| \leq m$.

\begin{lemma}
    \label{lem:reduction-to-finite}
    Let $G$ be a group.
    \begin{enumerate}[(1)]
        \item $G$ is $m$-by-class-$s$-by-$n$ if and only if every finitely generated subgroup of $G$ is so.
        \item If $G$ is residually finite, $G$ is $m$-by-class-$s$-by-$n$ if and only if every finite quotient of $G$ is so.
    \end{enumerate}
\end{lemma}
\begin{proof}
    (1) Suppose $[G:H] \leq n$ and $|\gamma_s(H)| \leq m$.
    Let $\Gamma \leq G$ be a subgroup.
    Then clearly $[\Gamma : H \cap \Gamma] \leq n$ and $|\gamma_s(H \cap \Gamma)| \leq m$,
    so the forward implication is clear.
    Now suppose every finitely generated subgroup $\Gamma \leq G$ is $m$-by-class-$s$-by-$n$.
    Let $I$ be the set of finitely generated subgroups of $G$.
    Since any $d$-generated group has at most $n!^d$ subgroups of index at most $n$,
    the product space
    \[
        X = \prod_{\Gamma \in I} \{H \leq \Gamma : [\Gamma:H] \leq n, |\gamma_s(H)| \leq m\}
    \]
    is compact.
    For each $\Gamma \in I$ let $C_\Gamma$ be the set of all vectors $(H_\Delta : \Delta \in I) \in X$ such that
    $H_\Delta = H_\Gamma \cap \Delta$ for $\Delta \leq \Gamma$.
    Then the sets $C_\Gamma$ are closed, $C_{\Gamma_1} \cap \cdots \cap C_{\Gamma_k} \supset C_\Delta$ where $\Delta = \langle \Gamma_1, \dots, \Gamma_k\rangle$,
    and the hypothesis that every $\Gamma \in I$ is $m$-by-class-$s$-by-$n$ implies that $C_\Gamma \neq \emptyset$.
    Hence the sets $C_\Gamma$ have the finite intersection property, so by compactness $\bigcap_{\Gamma \in I} C_\Gamma \neq \emptyset$.
    If $(H_\Gamma) \in \bigcap_{\Gamma \in I} C_\Gamma$ then $H_\Gamma = H \cap \Gamma$ for all $\Gamma$, where $H = \langle H_\Gamma : \Gamma \in I\rangle$.
    Now $[\Gamma : H \cap \Gamma] \leq n$ for every $\Gamma \in I$, so $[G : H] \leq n$,
    and $|\gamma_s(H \cap \Gamma)| \leq m$ for every $\Gamma \in I$, so $|\gamma_s(H)| \leq m$.
    Hence $G$ is $m$-by-class-$s$-by-$n$, as claimed.
    
    (2) Suppose $[G:H] \leq n$ and $|\gamma_s(H)| \leq m$.
    Let $\bar G = G / N$ be a quotient of $G$.
    Then $[\bar G : \bar H] = [G : HN] \leq n$ and $|\gamma_s(\bar H)| \leq |\gamma_s(H)| \leq m$,
    so the forward implication is clear.
    Now suppose every finite quotient $Q = G / N$ is $m$-by-class-$s$-by-$n$.
    Let $I$ be the set of finite quotients of $G$.
%    Since any finite group has finitely many subgroups, the product space
%    \[
%        X = \prod_{Q \in I} \{ H \leq Q : [Q : H] \leq n, |\gamma_s(H)| \leq m \}
%    \]
%    is compact.
%    For each $Q \in I$ let $C_Q$ be the set of all vectors $(H_R : R \in I) \in X$ such that $H_R = \overline{H_Q}$ for each
%    quotient $R$ of $Q$.
%    Then the sets $C_Q$ are closed, nested, and nonempty by hypothesis,
%    so by compactness there is some $(H_Q : Q \in I) \in \bigcap_{Q \in I} C_Q$.
    A compactness argument as in the previous part establishes that there are subgroups $H_Q \leq Q$ for each $Q \in I$
    obeying $[Q : H_Q] \leq n$ and $|\gamma_s(H_Q)| \leq m$ and
    which are consistent in the sense that if $R$ is (naturally) a quotient of $Q$ then $H_R$ is the quotient of $H_Q$.
    Since $[Q : H_Q] \leq n$ for all $Q$ there is some $Q$ such that $[Q : H_Q]$ is maximal.
    Let $H$ be the preimage of $H_Q$ in $G$ for this $Q$.
    Then $[G : H] = [Q : H_Q] \leq n$
    and $H_{\bar G} = \bar H$ for every quotient $\bar G \in I$,
    so $|\gamma_s(H)| = \max_{\bar G \in I} |\gamma_s(\bar H)| \leq m$.
    Hence $G$ is $m$-by-class-$s$-by-$n$, as claimed.
\end{proof}

We will also need to know that the covering condition \eqref{eq:covering-condition}
behaves well with respect to subgroups and quotients.

\begin{lemma}
    \label{lem:covering-subgroups-and-quotients}
    Suppose $G$ satisfies the covering condition \eqref{eq:covering-condition}.
    \begin{enumerate}[(1)]
        \item If $H \leq G$ then $H$ satisfies the covering condition with $(n^2, S')$ in place of $(n, S)$
        for some set $S'$ of size at most $|S|$.
        \item If $\bar G$ is a homomorphic image of $G$ then $\bar G$ satisfies the covering condition with $(n, \bar S)$
        in place of $(n, S)$.
    \end{enumerate}
\end{lemma}
\begin{proof}
    (1) Let $S'$ be a set containing one point of $Bs \cap H$ whenever $Bs \cap H \neq \emptyset$.
    If $x \in \Comm(H, H) \subset \Comm(G, G)$ then $x \in Bs \cap H$, so $x, s' \in Bs$ for some $s' \in S'$.
    Write $x = b_1s$ and $s' = b_2s$.
    Then $x = b_1b_2^{-1}s' \in B^2 s$, so $\Comm(H, H) \subset B^2 S'$.
    
    (2) Clear.
\end{proof}

Somewhat related to this lemma, the beautiful result \cite{MTVV}*{Theorem~1.21} states that, for $G$ finite,
\begin{equation}
    \label{eq:dk-submult}
    d_k(G) \leq d_k(N) d_k(G/N)
\end{equation}
whenever $N \nsgp G$.
This shows that $d_k$ behaves very well with respect to normal subgroups and quotients.
However, we do not need a result of this form because we work mainly with the covering condition instead of $d_2$.

%Finally, we will need to know that the structure of the map $T$ described in \Cref{thm:d2-theorem}
%can be detected in finite sections.
%
%\begin{lemma}
%    Let $H$ be a nilpotent group of class at most 3
%    such that in every finite section $\Sigma = K/N$ (where $N \nsgp K \leq H$)
%    the triple commutator $T_\Sigma = [x, y, z] : (\Sigma/\Sigma')^3 \to \gamma_3(\Sigma)$
%    as a trilinear map of abelian groups has the form
%    \[
%        T_\Sigma(x, y, z) = A_\Sigma(a_\Sigma(x, y), z) + B_\Sigma(b_\Sigma(x, z), y) + C_\Sigma(c_\Sigma(y, z), x),
%    \]
%    where the codomains of $a, b, c$ are at most $m$.
%    Then the triple commutator $T : (H / H')^3 \to \gamma_3(H)$ also has this form.
%\end{lemma}
%\begin{proof}
%    Since $H$ is nilpotent of class 3, it is locally residually finite.
%    By a compactness argument such as in the proof of \Cref{lem:reduction-to-finite},
%    we may assume that the maps $a_\Sigma, b_\Sigma, c_\Sigma$ have constant codomains as $\Sigma$ ranges over the finite sections of $H$
%    and in fact lift to maps defined on $H$.
%    We may also assume $A_\Sigma, B_\Sigma, C_\Sigma$ are consistent in the sense that
%    if $\Sigma_1$ is naturally a section of $\Sigma_2$ then $A_{\Sigma_1}$ is obtained from $A_{\Sigma_2}$ by restriction,
%    and similarly for $B_\Sigma$ and $C_\Sigma$.
%    Thus we have maps $A, B, C$ such that
%    \[
%        T(x, y, z) = A(a(x, y), z) + B(b(x, z), y) + C(c(y, z), x),
%    \]
%    with only the caveat that $A, B, C$ take values in the profinite completion of $\gamma_3(H)$.
%\end{proof}

\subsection{Multilinear bias and a 4-linear rank-reduction lemma}
%\label{sec:rank-reduction}

We need the following structure theorem for biased multilinear maps.
If $A_1, \dots, A_k$ are groups and $I \subset [k] = \{1, \dots, k\}$ we write $A_I = \prod_{i \in I} A_i$
and we write $x_I$ for the projection of $x \in A_{[k]}$ to $A_I$.
We write $I^c$ for the complement $[k] \setminus I$.
If $g$ is a (multilinear) function we write $\cod(g)$ for its codomain.

\begin{theorem}[\cite{bias-notes}*{Corollary~1.3}]
    \label{thm:bias-structure}
    Suppose $A_1, \dots, A_k, B$ are finite abelian groups and $F : A_{[k]} \to B$ is a multilinear map such that $\P(F = 0) \geq \eps > 0$.
    Then there is an expression
    \begin{equation}
        \label{eq:biased-structured}
        F(x) = \sum_{\emptyset \neq I \subset [k]} G_I(g_I(x_I), x_{I^c}),
    \end{equation}
    where for each $I$ the functions $g_I$ and $G_I$ are multilinear maps
    \begin{align*}
        &g_I : A_I \to \cod(g_I), \\
        &G_I : \cod(g_I) \times A_{I^c} \to B,
    \end{align*}
    and $|\cod(g_I)|$ is $(\eps, k)$-bounded.
\end{theorem}

If we call
\[
    \prod_{\emptyset \neq I \subset [k]} |\cod(g_I)|
\]
(or perhaps its logarithm)
the \emph{rank} of the expression \eqref{eq:biased-structured},
the conclusion of \Cref{thm:bias-structure} is simply that there is an expression for $F$ of $(\eps, k)$-bounded rank.
This terminology is reasonable given the connection to the theory of analytic rank,
which is explained in \cite{bias-notes}.

We need to probe the uniqueness of bounded-rank expression such as \eqref{eq:biased-structured}.
Certainly the expression need not be unique, strictly speaking, but suppose $F$ has two bounded-rank expressions
\[
    F(x) = \sum_{I \in \II} G_I(g_I(x_I), x_{I^c})
    = \sum_{J \in \JJ} G_J(g_J(x_J), x_{J^c})
\]
for some sets $\II$ and $\JJ$ of nonempty subsets of $[k]$.
If $\II\neq \JJ$, then it is natural to ask whether there is a third expression, still of bounded rank,
only involving terms which could have appeared in either sum.
This may be true generally.
We will establish it in the special case we need.

\begin{lemma}
    \label{lem:rank-refinement}
    Suppose $F : A_1 \times A_2 \times A_3 \times A_4 \to B$ is a multilinear map of abelian groups with expressions
    \[
        F(x) = \sum_{I \in \II} F_I(f_I(x_I), x_{I^c})
        = \sum_{J \in \JJ} G_J(g_J(x_J), x_{J^c}),
    \]
    both of complexity at most $C$, where
    \begin{align*}
        \II &= \{\{1\}, \{2\}, \{3\}, \{4\}, \{1, 2\}, \{1, 3\}, \{2, 3\}, \{1, 2, 3\}\},\\
        \JJ &= \{\{1\}, \{2\}, \{3\}, \{4\}, \{1, 2\}, \{1, 4\}, \{2, 4\}, \{3, 4\}, \\
        &\qquad \{1, 2, 4\}, \{1, 3, 4\}, \{2,3,4\}, \{1,2,3,4\}\}.
    \end{align*}
    Then there are $C$-bounded-index subgroups $A_i' \leq A_i$ such that the restriction $F'$ of $F$ to ${A_1' \times A_2' \times A_3' \times A_4'}$ has the form
    \[
        F'(x, y, z, w) = F_1(f_1(x, y), z, w) + F_2(f_2(x, z), f_2'(y, w)) + F_3(f_3(x, w), f_3'(y, z)),
    \]
    where $F_1, F_2, F_3, f_1, f_2, f_2', f_3, f_3'$ are multilinear and $f_1, f_2, f_2', f_3, f_3'$ have $C$-bounded codomains.
\end{lemma}

We will continue to use large Roman letters $F, G, H, \dots$ for arbitrary multilinear maps
and little Roman letters $f, g, h, \dots$ for multilinear maps with bounded codomain.

\begin{proof}
    Throughout the proof, ``bounded'' means ``$C$-bounded''.
    We will repeatedly use the observation that if we have boundedly many maps $g_1, \dots, g_n$ of one variable and bounded codomain
    then we can pass to the kernel of all these maps to get rid of them.
    In particular we can immediately get rid of all terms of the forms
    \[
        G(g(x), y, z, w), G(x, g(y), z, w), G(x, y, g(z), w), G(x, y, z, g(w)),
    \]
    so forget those.
    Hence we may assume the $\II$-expression takes the simpler form
    \begin{equation}
        \label{eq:I-expr}
        \begin{aligned}
        F(x, y, z, w) = F_1(f_1(x, y), z, w) + F_2(f_2(x, z), y, w) + F_3(f_3(y, z), x, w) \\
        + F_4(f_4(x, y, z), w).
        \end{aligned}
    \end{equation}
    
    First consider the term $F_2(f_2(x, z), y, w)$.
    Fix $u = f_2(x, z)$.
    By equating \eqref{eq:I-expr} with a $\JJ$-expression we find an identity of the form
    \begin{equation}
        \label{eq:u-expr}
        F_2(u, y, w) = G_u(g_u(y, w)) + G_u'(g_u'(y), w) + G_u''(g_u''(w), y),
    \end{equation}
    where $g_u, g_u', g_u''$ have bounded codomain.
    Since $|\cod(f_2)|$ is bounded, in fact we have \eqref{eq:u-expr} for every $u$ in the group
    \[
        U = \langle f_2(x, z) : x \in A_1, z \in A_3\rangle \leq \cod(f_2).
    \]
    Pass to $\bigcap_{u \in U} \ker g_u' \cap \ker g_u''$ to reduce to
    \[
        F_2(u, y, w) = G_u(g_u(y, w)).
    \]
    Now replace $g_u$ with $f_2'' = (g_u : u \in U)$, which has bounded codomain $\prod_{u \in U} \cod(g_u)$,
    and replace $G_u$ with $G_u \circ \pi_u$, where $\pi_u$ is the projection; thus
    \[
        F_2(u, y, w) = G_u(f_2''(y, w)).
    \]
    We may restrict the codomain of $f_2''$ to $\langle f_2''(y, w) : y \in A_2, w \in A_4\rangle$
    and we may restrict the domain of $G_u$ to $\cod(f_2'')$.
    But now since $F_2(u, y, w)$ is $u$-linear, $G_u(v)$ must be $u$-linear for every $v \in \cod(g)$, so
    \[
        F_2(u, y, w) = F_2'(u, f_2''(y, w))
    \]
    for some multilinear function $F_2'$.
    Hence
    \[
        F_2(f_2(x, z), y, w) = F_2'(f_2(x, z), f_2''(y, w)),
    \]
    and by an analogous argument
    \[
        F_3(f_3(y, z), x, w) = F_3'(f_3(y, z), f_3''(x, w)).
    \]
    
    Finally, consider $F_4(f_4(x, y, z), w)$.
    By comparing the $\II$- and $\JJ$-expressions, we get an identity of the form
    \begin{equation}
        \label{eq:xyz|w}
        F_4(f_4(x, y, z), w) = H(h(x, y), z, w) + H'(x, y, z, w),
    \end{equation}
    where $H'(x, y, z, w)$ has an expression involving terms only of the forms
    \begin{align*}
        & G(g(x, z), g'(y, w)), G(g(y, z), g'(x, w)), \\
        & G(g(x, w), y, z), G(g(y, w), x, z), G(g(z, w), x, y), \\
        & G(g(x, y, w), z), G(g(x, z, w), y), G(g(y, z, w), x), G(g(x, y, z, w)).
    \end{align*}
    Fix $x, y$ and let $u = h(x, y)$. Then by rearranging \eqref{eq:xyz|w} we get an identity of the form
    \[
        H(u, z, w) = G_u(g_u(z, w)) + G'_u(g'_u(z), w) + G''_u(g'_u(w), z).
    \]
    By arguing exactly as before we establish that
    \[
        H(h(x, y), z, w) = G(h(x, y), g(z, w)).
    \]
    Now fix $x, y, z$ and let $u = g(x, y, z)$.
    Then from \eqref{eq:xyz|w} we have an expression of the form
    \[
        F_4(u, w) = H_u(h_u(w)).
    \]
    Let $V = \langle g(x, y, z) : (x, y, z) \in A_1 \times A_2 \times A_3 \rangle$.
    Passing to $\bigcap_{u \in V} \ker h_u$, $F_4(u, w) = 0$.
    Thus the $\II$-expression \eqref{eq:I-expr} is reduced to the desired form.
\end{proof}

\section{Neumann's theorem for metric entropy}
\label{sec:metric-neumann}

Let $G$ be a group.
An \emph{invariant seminorm} on $G$ is a function $\|\cdot\| : G \to [0, \infty]$ satisfying identically
\begin{enumerate}[(1)]
    \item (reflexivity) $\|1\| = 0$,
    \item (symmetry) $\|x^{-1}\| = \|x\|$,
    \item (invariance) $\|x^y\| = \|x\|$,
    \item (triangle inequality) $\|xy\| \leq \|x\| + \|y\|$.\footnote{For our purposes it would suffice to know that $\|xy\|$ is always $(\|x\|, \|y\|)$-bounded.}
\end{enumerate}
If $\|\cdot\|$ is an invariant seminorm then $d(x, y) = \|xy^{-1}\|$ defines a bi-invariant pseudometric on $G$.
Conversely if $d$ is a bi-invariant pseudometric then $\|x\| = d(x, 1)$ is an invariant seminorm.
This makes it reasonable to use metric language.

\begin{example}
    The \emph{discrete norm} is
    \[
        \|x\| = \begin{cases}
            0 &: x = 1 \\
            \infty &: x \neq 1
        \end{cases}.
    \]
\end{example}

\begin{example}
    The \emph{conjugacy class norm} is
    \[
        \|x\| = \log |x^G|.
    \]
\end{example}

\begin{theorem}
    \label{thm:metric-neumann}
    Let $G$ be a finite group with an invariant seminorm $\|\cdot\|$.
    Let $A, B \nsgp G$ be normal subgroups
    and suppose $\P_{a \in A, b \in B}(\|[a, b]\| \leq C) \geq 1/C$.
    Then there are subgroups $H \leq A$ and $K \leq B$,
    both normal in $G$,
    such that $[A:H]$ and $[B:K]$ are $C$-bounded
    and such that $\Comm(H, K)$ is covered by $C$-boundedly many balls of $C$-bounded radius.
    
    Conversely, if there are subgroups $H \leq A$ and $K \leq B$
    such that ${[A:H]},{[B:K]} \leq C$
    and such that $\Comm(H, K)$ is covered by at most $C$ balls of radius $C$,
    then there is some $C$-bounded number $D$ such that $\P_{a \in A, b \in B}(\|[a, b]\| \leq D) \geq 1/D$.
\end{theorem}

We will use the following lemma from \cite{eberhard} multiple times.

\begin{lemma}[See \cite{eberhard}*{Lemma~2.1}]
    \label{lem:bounded-gen}
    Let $G$ be a finite group and $X$ a symmetric subset of $G$ containing the identity.
    Then $\langle X \rangle = X^r$ for $r = 3\floor{|G|/|X|}$.
\end{lemma}

\begin{proof}[Proof of \Cref{thm:metric-neumann}]
    Let $X = \{a \in A: \P_{b \in B}(\|[a, b]\| \leq C) \geq 1/2C\}$. Then
    \[
        1/C \leq \P_{a \in A, b \in B}(\|[a, b]\| \leq C) \leq \P(a \in X) + 1/2C,
    \]
    so $\P(a \in X) \geq 1/2C$.
    Fix $x \in X$.
    Let $Y_x = \{b \in B : \|[x, b]\| \leq C\}$.
    Then $\P_{b \in B}(b \in Y_x) \geq 1/2C$.
    It follows that $K_x = \langle Y_x\rangle$ has $C$-bounded index in $B$
    and every $k \in K_x$ is the product of $C$-boundedly many $y \in Y_x$ by \Cref{lem:bounded-gen}.
    Suppose $k \in K_x$, so $k = y_1 \cdots y_n$ for some $C$-bounded $n$ and $y_1, \dots, y_n \in Y_x$.
    Then
    \[
        [x, k] = [x, y_1 \cdots y_n] = [x, y_n] [x, y_{n-1}]^{y_n} \cdots [x, y_1]^{y_2 \cdots y_n},
    \]
    so $\|[x, k]\| \leq nC$.
    Hence $\|x, k]\|$ is $C$-bounded for every $x \in X$ and $k \in K_x$.
    
    Now let $H = \langle X \rangle$.
    Since $X$ is a normal subset of $G$, $H$ is normal in $G$.
    Since $\P(a \in X) \geq 1/2C$, $[A : H] \leq 2C$,
    and if $h \in H$ then $h = x_1 \cdots x_n$ for some $x_1, \dots, x_n \in X$ and some $C$-bounded $n$ by \Cref{lem:bounded-gen} again.
    Then the subgroup $K_h = K_{x_1} \cap \cdots \cap K_{x_n}$ has $C$-bounded index and every $k \in K_h$
    is such that $\|[h, k]\|$ is $C$-bounded.
    
    Symmetrically, there is a subgroup $K \leq B$ of index at most $2C$
    such that for every $k \in K$ there is a subgroup $H_k \leq A$ of $C$-bounded index and such that
    $\|[h, k]\|$ is $C$-bounded for every $h \in H_k$.
    
    Hence there is a $C$-bounded number $D$ such that for every $h \in H$ we have
    \[
        \P_{k \in K}(\|[h, k]\| \leq D) \geq 1/D
    \]
    and for every $k \in K$ we have
    \[
        \P_{h \in H}(\|[h, k]\| \leq D) \geq 1/D.
    \]
    Let $(h, k) \in H \times K$.
    Then if $y \in K$ and $\|[h, y]\| \leq D$ we have
    \[
        [h, y k] = [h, k] [h, y]^k,
    \]
    so, letting $k' = yk$,
    \[
        \|[h, k'] [h, k]^{-1}\| \leq D.
    \]
    Similarly if $x \in H$ and $\|[x, k']\| \leq D$ we have
    \[
        [x h, k'] = [x, k']^h [h, k'],
    \]
    so, letting $h' = xh$,
    \[
        \|[h', k'] [h, k]^{-1}\| \leq \|[h', k'] [h, k']^{-1}\| + \|[h, k'] [h, k]^{-1}\| \leq 2D.
    \]
    Hence
    \[
        \P_{h' \in H, k' \in K}(\|[h', k'] [h, k]^{-1}\| \leq 2D) \geq 1/D^2.
    \]
    It follows that if $S$ is any set of $(4D+1)$-separated points of $\Comm(H, K)$
    we must have $|S| \leq D^2$.
    Hence there is a set $S$ of size at most $D^2$ such that any point of $\Comm(H, K)$ is within distance $4D+1$ of a point of $S$.
    
    For the converse, write $B$ for the ball of radius $C$ and $S$ for a set of size at most $C$ such that
    \[
        \Comm(H, K) \subset BS.
    \]
    Let $h \in H$.
    By the pigeonhole principle there is some $s \in S$ such that
    \[
        \P_{k \in K}([h, k] \in Bs) \geq 1/|S| \geq 1/C.
    \]
    Fix $k_0$ such that $[h,k_0] \in Bs$.
    Then whenever $[h, k] \in Bs$ we have
    \[
        [h,kk_0^{-1}] = \br{
            [h, k_0]^{-1} [h, k]
        }^{k_0^{-1}}
        \in B^2,
    \]
    so $\|[h, kk_0^{-1}]\| \leq 2C$.
    Thus
    \[
        \P_{k \in K}(\|[h, k]\| \leq 2C) \geq 1/C.
    \]
    This holds for every $h \in H$, so
    \[
        \P_{a \in A, b \in B}(\|[a, b]\| \leq 2C) \geq \frac1{C [A:H] [B:K]} \geq 1/C^3.
    \]
    Letting $D = 2C^3$, we get the claim.
\end{proof}

By taking the discrete norm we recover the asymmetric version of Neumann's theorem
stated earlier as \Cref{thm:asymmetric-neumann}.

\begin{corollary}[Neumann's theorem, asymmetric version]
    Let $G$ be a finite group with normal subgroups $A, B \nsgp G$
    such that $\P_{a \in A, b \in B}([a, b] = 1) \geq 1/C$.
    Then there are $C$-bounded-index subgroups $H \leq A$ and $K \leq B$,
    both normal in $G$,
    such that $|[H, K]|$ is $C$-bounded.
\end{corollary}
\begin{proof}
    By \Cref{thm:metric-neumann} applied to the discrete norm,
    there are subgroups $H \leq A$ and $K \leq B$ of $C$-bounded index such that $\Comm(H, K)$ has $C$-bounded size.
    Now we may apply \cite{DS}*{Lemma~2.1} to conclude that $|[H, K]|$ is $C$-bounded,
    or we may argue directly as follows.
    Let $[h_1, k_1], \dots, [h_n, k_n]$ be an enumeration of the elements of $\Comm(H, K)$.
    Let $H_0 = \langle h_1, \dots, h_n, \Comm(H, K)\rangle$, $K_0 = \langle k_1, \dots, k_n, \Comm(H, K) \rangle$.
    Then $H_0, K_0 \nsgp \langle H_0, K_0\rangle$ and $[H_0, K_0] = [H, K]$.
    Since both $H_0$ and $K_0$ have boundedly many generators
    and $\Comm(H, K)$ is bounded,
    the centralizers $C_{H_0}(K_0)$ and $C_{K_0}(H_0)$ have bounded indices in $H_0$ and $K_0$, respectively.
    Now it follows from Baer's asymmetric version of Schur's theorem (\cite{robinson}*{14.5.2}) that $[H_0, K_0] = [H, K]$ has $C$-bounded size.
\end{proof}

The main corollary we need in the rest of the paper is the conjugacy class norm case.

\begin{corollary}
    \label{cor:d2-neumann}
    Let $G$ be a finite group such that $d_2(G) \geq 1/C$.
    Then there is a subgroup $H \leq G$ of $C$-bounded index satisfying the covering condition \eqref{eq:covering-condition}
    with $n$ and $|S|$ both $C$-bounded.
\end{corollary}
\begin{proof}
    By \Cref{thm:metric-neumann} applied with $A = B = G$ and the conjugacy class norm $\|x\| = \log |x^G|$,
    there is a subgroup $H$ of $C$-bounded index such that $\Comm(H, H)$ is covered by $C$-boundedly many balls of $C$-bounded radius.
    Let $S$ be a set containing one point from each ball.
    Then $|S|$ is $C$-bounded and for every $x, y \in H$ there is some $s \in S$ such that $\|[x, y] s^{-1}\|$ is $C$-bounded,
    so $[x, y] s^{-1}$ has $C$-boundedly many conjugates.
\end{proof}

\section{From covering to bounded derived length}
\label{sec:to-bounded-derived-length}

Throughout this section we assume the covering condition \eqref{eq:covering-condition}, which states
\[
    \Comm(G, G) \subset B S,
\]
where $\Comm(G, G)$ is the set of commutators of $G$, $B = \{x \in G: |x^G| \leq n\}$, and $S \subset G'$ is finite.
We will show that $G$ is bounded-by-derived-length-4-by-bounded, with a bound depending only on $|S|$ and $n$.

\begin{proposition}
    \label{prop:bounded-derived-length}
    Let $G$ be a group satisfying the covering condition~\eqref{eq:covering-condition}.
    Then $G$ has a subgroup $H$ such that $[G:H]$ and $|H^{(4)}|$ are both $(|S|,n)$-bounded.
    In particular if $K = C_H(H^{(4)})$ then $[G:K]$ is $(|S|, n)$-bounded and $K^{(5)} = 1$.
\end{proposition}

\begin{remark}
    \leavevmode
    \begin{enumerate}[(1)]
        \item If $S = \{1\}$, the covering condition states that the commutators in $G$ have boundedly many conjugates.
        The main result of \cite{dieshu} states in this case that $G$ is bounded-by-metabelian.
        We closely follow some parts of \cite{dieshu} to prove the more general but weaker result stated above.
        \item The proof actually shows that $G$ is bounded-by-metabelian-by-class-2-by-bounded.
        We do not need this more detailed information right now, and we will eventually prove the stronger result \Cref{thm:covered-groups} anyway.
        \item Results of Bors, Larsen, and Shalev (see \cite{bors-shalev}*{Corollary~1.5} and references),
        relying on the classification of finite simple groups,
        imply that if $G$ is finite and $d_k(G) \geq \eps > 0$ then
        $[G : \rad(G)]$ is $(k, \eps)$-bounded, where $\rad(G)$ is the soluble radical of $G$,
        and it follows from \eqref{eq:dk-submult} that $\rad(G)$ has $(k, \eps)$-bounded derived length.
        By comparison, in \Cref{prop:bounded-derived-length}, we do not depend on the classification,
        the group $G$ is allowed to be infinite, and there is a stronger conclusion (though it will be superseded by \Cref{thm:covered-groups}).
    \end{enumerate}
\end{remark}

Now assume the hypothesis of \Cref{prop:bounded-derived-length}.
We may assume $1 \in S$ and we will use inducton on $|S|$.
Suppose there is some commutator $x \in \Comm(G, G)$ such that $n < |x^G| \leq n^{100}$.
By the covering condition $x \in Bs$ for some nontrivial $s \in S$.
This implies that $|s^G| \leq n^{101}$.
Then for any $g \in Bs$ we have $|g^G| \leq n^{102}$.
Hence we can remove $s$ from $S$ at the cost of increasing $n$ to $n^{102}$
and then apply induction on $|S|$.
Hence we may assume $B$ satisfies the following stability condition:
\begin{equation}
    \label{eq:B-stability}
    x \in \Comm(G, G) \wedge |x^G| \leq n^{100} \implies x \in B.
\end{equation}

Let $H = \langle B \rangle$.
Given an element $g\in H$, we write $l(g)$ for the minimal number $l$ with the property that $g$ can be written as a product of $l$ elements of $B$.

\begin{lemma}
    [See \cite{dieshu}*{Lemma~2.1}]
    \label{21}
    Let $K\leq H$ be a subgroup of index $m$ in $H$, and let $b\in H$. Then the coset $Kb$ contains an element $g$ such that $l(g)\leq m-1$.
\end{lemma}

\begin{lemma}\label{23}
    For any $x\in B$ the subgroup $[H,x]$ has $n$-bounded order.
    Consequently, the order of $\langle[H,x]^G\rangle$ is $n$-bounded.
\end{lemma}
\begin{proof}
    Let $m = [H : C_H(x)] = |x^H|$, so $m \leq n$.
    By \Cref{21} we can choose $b_1, \dots, b_m$ such that $x^H = \{x^{b_1}, \dots, x^{b_m}\}$
    and $l(b_i) \leq m$ for each $i$, so $|b_i^G| \leq n^m$.
    Note then that $[H, x]$ is generated by $[b_1, x], \dots, [b_m, x]$.
    Let $T = \langle x, b_1, \dots, b_m\rangle$ and $U = C_G(T)$.
    Then since $Z(T) = T \cap C_G(T)$ we have
    \[
        [T : Z(T)] \leq [G : C_G(T)] \leq n^{m^2+1}.
    \]
    By Schur's theorem, $T'$ is $n$-bounded.
    Since $[H, x] \leq T'$, this gives the first statement.
    
    For the second statement, recall that $[H, x] \nsgp H \nsgp G$.
    The $G$-conjugates of $[H, x]$ are $[H, x^{b_1}], \dots, [H, x^{b_m}]$,
    so $\langle [H, x]^G \rangle = \prod_{i=1}^m [H, x^{b_i}]$, so $|\langle [H, x]^G\rangle| \leq |[H, x]|^m$.
\end{proof}

Now let $X = B \cap \Comm(G, G)$
and pick $x \in X$ so that $m = |x^H|$ is maximal.
Define $b_1, \dots, b_m, T, U$ as in the proof of \Cref{23}.

\begin{lemma}\label{24}
    If $u\in U$ and $ux\in X$, then $[H,u]\leq[H,x]$.
\end{lemma}
\begin{proof}
    Since $ux \in X$, $|(ux)^H| \leq m$.
    For each $i=1,\dots,m$ we have $(ux)^{b_i} = ux^{b_i}$.
    Since these elements are distinct it follows that
    \[
        (ux)^H = \{ux^{b_i} : i = 1, \dots, m\}.
    \]
    Thus for any $h \in H$ there is $b \in \{b_1, \dots, b_m\}$ such that $(ux)^h = u^h x^h = ux^b$.
    It follows that $[u, h] = u^{-1} u^h = x^b x^{-h} \in [H, x]$, so the claim holds.
\end{proof}

Let $R=\langle[H, x]^G\rangle$.
By \Cref{23} the order of $R$ is bounded.
By \Cref{24}, $[H, u] \leq R$ whenever $u \in U$ and $ux \in X$.

Let $U_0$ be the normal core of $U$.
Then $[G : U_0]$ is $n$-bounded.
Let $B_0 = B^2 \cap U_0$.

Now write $x=[x_1 ,x_2]$ where $x_1, x_2\in G$.
Let $t_1, t_2 \in B_0$ and consider
\[
    y = [x_1 t_1, x_2 t_2].
\]
We may expand $y$ has a product of four commutators:
\[
    y
    = [x_1, t_2]^{t_1} [x_1, x_2]^{t_2 t_1} [t_1, t_2] [t_1, x_2]^{t_2}.
\]
Since $[x_1, x_2] = x$ and $[x, U_0] = 1$, we may simplify this slightly to
\[
    y
    = x [x_1, t_2]^{t_1} [t_1, t_2] [t_1, x_2]^{t_2}.
\]
Each commutator involving a $t_i$ is in $B^4$, and $x \in B$, so $y \in B^{13}$.
By \eqref{eq:B-stability}, this implies $y \in B$.
Applying \Cref{24}, we find that
\[
    [H, [x_1, t_2]^{t_1} [t_1, t_2] [t_1, x_2]^{t_2}] \leq R,
\]
for all $t_1, t_2, \in B_0$.
Taking $t_2 = 1$ we get $[H, [t_1, x_2]] \leq R$.
Taking $t_1 = 1$ we get $[H, [x_1, t_2]] \leq R$.
Hence $[H, [t_1, t_2]] \leq R$.
Since this holds for all $t_1, t_2 \in B_0$ it follows that
\[
    [H, \langle B_0 \rangle'] \leq R.
\]
Also $[t_1, t_2] \in B$ by \eqref{eq:B-stability},
so $\langle B_0 \rangle' \leq H$.
Hence $\langle B_0 \rangle$ is metabelian mod $R$.

Next suppose $y_1, y_2 \in \Comm(U_0, U_0) \cap Bs$ for some $s \in S$.
Then $y_1 y_2^{-1} \in B^2 \cap U_0 = B_0$.
Hence there are at most $|S|$ commutators in $U_0/\langle B_0\rangle$,
so $(U_0 / \langle B_0 \rangle)'$ has bounded order by Neumann's BFC theorem (\Cref{thm:BFC}).
We can pass to a bounded-index subgroup of $U_0$ whose image in $U_0 / \langle B_0\rangle$ is nilpotent of class at most 2.
Hence $G$ is bounded-by-metabelian-by-class-2-by-bounded.
This proves \Cref{prop:bounded-derived-length}.

\section{From bounded derived length to bounded class}
\label{sec:to-bounded-class}

Let $G$ be a soluble group satisfying the covering condition \eqref{eq:covering-condition}.
In this section we show that $G$ has a nilpotent subgroup of bounded index and class, with bounds depending on $|S|$, $n$, and the derived length $l$ of $G$.

\begin{proposition}
    \label{prop:bounded-class}
    Let $G$ be a group satisfying the covering condition~\eqref{eq:covering-condition}.
    Assume $G$ is soluble of derived length $l$.
    Then $G$ has a nilpotent subgroup $H$ such that $[G:H]$ and the class of $H$ are both $(|S|, n, l)$-bounded.
\end{proposition}

The proof of this proposition fills the rest of this section.
Throughout, we say ``bounded'' to mean ``$(|S|, n, l)$-bounded''.

\subsection{Reduction to the finite metabelian case}

Recall Hall's criterion for nilpotency (see \cite{robinson}*{5.2.10}): if $N$ is a normal subgroup of $G$ such that $N$ and $G/N'$ are both nilpotent
then $G$ is nilpotent with class bounded in terms of the classes of $N$ and $G/N'$.

By induction on derived length, $G'$ has a nilpotent subgroup $A$ of bounded index and class.
Replacing $A$ with its normal core in $G'$, we may assume $A$ is normal in $G'$.
Then it follows from Fitting's theorem that the characteristic closure of $A$ has bounded class,
so we may assume $A$ is characteristic in $G'$.
If $A' \neq 1$ then by induction on the class of $A$ there is a nilpotent subgroup of $G / A'$ of bounded index and class.
By Hall's criterion we are done in this case, so we may assume $A$ is abelian.
Since $|G'/A|$ is bounded, we may replace $G$ with a bounded-index subgroup whose image in $G/A$ has class at most 2.

In particular, $G$ is abelian-by-class-2.
By another classical result of Hall, any abelian-by-nilpotent group is locally residually finite (see \cite{robinson}*{15.4.1}).
Applying \Cref{lem:reduction-to-finite,lem:covering-subgroups-and-quotients}, we may assume $G$ is finite.

Let $g \in G$ and assume $g$ is central modulo $A$.
Then the subgroup $[A, g]$ is normal in $G$ and consists of commutators, so $[A, g] \subset BS$.
Note also that $a \mapsto [a, g]$ is a homomorphism $A \to [A, g]$.
Suppose $a, b \in A$ satisfy $[a, g], [b, g] \in Bs$ for some $s \in S$.
Then
\[
    [ab^{-1}, g] = [a, g] [b, g]^{-1} \in B^2.
\]
It follows that
\[
    \P_{a \in A}([a, g] \in B^2) \geq 1/|S|,
\]
so
\[
    \P_{a \in A, h \in G}([a, g, h] = 1) \geq \frac1{|S| n^2}.
\]
Applying \Cref{thm:asymmetric-neumann}, there are subgroups $B_g \leq [A, g]$ and $T_g \leq G$, both normal in $G$,
such that the indices $[[A, g]: B_g]$ and $[G:T_g]$ and the order of $[B_g,T_g]$ are bounded.

Write $G' = \langle b_1, \dots, b_s \rangle A$, where $s$ is bounded.
Since $G/A$ has class 2, $b_1, \dots, b_s$ are central mod $A$.
Write $T_i$ for $T_{b_i}$, and $B_i$ for $B_{b_i}$.
Then $[A,G'] = \prod_{i=1}^s [A, b_i]$.
Let $U = \bigcap_{i=1}^s T_i \cap C_G([B_i, T_i]) \cap C_G([A, b_i] / B_i)$.
Then $[G:U]$ is bounded and $U'$ has class at most 4, since
\begin{align*}
    & \gamma_3(U) \leq A && (\text{since $G/A$ has class 2}), \\
    & [\gamma_3(U), U'] \leq [A, G'] = \prod_{i=1}^s [A, b_i], \\
    & [\gamma_3(U), U', U] \leq \prod_{i=1}^s B_i && (\text{since $U \leq C_G([A, b_i] / B_i$)}),\\
    & [\gamma_3(U), U', U, U] \leq \prod_{i=1}^s [B_i, T_i] && (\text{since $U \leq T_i$}),\\
    & [\gamma_3(U), U', U, U, U] = 1 && (\text{since $U \leq C_G([B_i, T_i]$)}),
\end{align*}
and
\[
    \gamma_5(U') \leq [\gamma_4(U), U', U, U, U, U] = 1.
\]
If we can show that $U / U''$ has a subgroup $H / U''$ of bounded index and class then
$H$ has bounded index in $G$ and it would follow from Hall's criterion
that $H$ has bounded class.
Thus \Cref{prop:bounded-class} is reduced to the finite metabelian case.

For the rest of this section we assume $G$ is a finite metabelian group and $A = G'$.
We will continue to refer to the subgroups $B_g \leq [A, g]$ and $T_g \leq G$,
now valid for any $g \in G$ since $G / A$ is abelian.

\subsection{The case of $p$-groups}

We use the notation $[x, {}_l y]$ for the long commutator $[x, y, \dots, y]$ where $y$ is repeated $l$ times.
If $l = 0$ then $[x, {}_l y] = x$.
An element $y$ of a group $\Gamma$ is \emph{$l$-Engel} if $[x, {}_l y] = 1$ for all $x \in \Gamma$,
and $\Gamma$ is \emph{$l$-Engel} if every $y \in \Gamma$ is $l$-Engel.
The following lemma will be useful.

\begin{lemma}\label{lll}
    Let $\Gamma$ be a metabelian group and suppose that $a,b\in \Gamma$ are $l$-Engel elements. If $x\in\langle a,b\rangle$, then $x$ is $(2l+1)$-Engel.
\end{lemma}
\begin{proof}
    Both subgroups $\Gamma'\langle a\rangle$ and $\Gamma'\langle b\rangle$ are normal in $\Gamma$ and nilpotent of class at most $l$.
    By Fitting's theorem, $\Gamma'\langle a,b\rangle$ is nilpotent of class at most $2l$ and the lemma follows.
\end{proof}

Now let $G$ be a finite metabelian $p$-group obeying the covering condition \eqref{eq:covering-condition}.
Let $A = G'$ and $B_g \leq [A, g]$ and $T_g \leq G$ be as in the previous subsection.
Since $|[B_g, T_g]|$ is bounded, there is a bounded $i$ such that $Z_i(G)$ contains $[B_g, T_g]$ for every $g\in G$.
Passing to the quotient $G/Z_i(G)$, we may assume $[B_g, T_g] = 1$ identically.
We may also assume that $B_g = C_{[A, g]}(T_g)$ and $T_g = C_G(B_g)$.

\begin{lemma}
    \label{engel}
    If $G$ is $l$-Engel then $G$ is nilpotent of $(|S|,n,l)$-bounded class.
\end{lemma}
\begin{proof}
    Since $G/T_g$ has bounded order, $G = \langle x_1, \dots, x_k\rangle T_g$ for some $x_1, \dots, x_k \in G$ and bounded $k$.
    Since $G$ is metabelian and $l$-Engel, $A\langle x_i\rangle$ has class at most $l$ for each $i$,
    so $A \langle x_1, \dots, x_k \rangle$ has bounded class by Fitting's theorem.
    In particular $B_g \langle x_1, \dots, x_k \rangle$ has bounded class.
    Since $[B_g, T_g] = 1$, this shows that $B_g$ is contained in $Z_j(G)$ for some bounded $j$.
    Since $[A,g] / B_g$ has bounded order, $[A, g]$ is contained in $Z_{j'}(G)$ for bounded $j'$.
    Since $[A, g] \in Z_{j'}(G)$ for all $g$, $\gamma_3(G) \leq Z_{j'}(G)$, so $G$ has class at most $j'+2$.
\end{proof}

Thus it suffices to find a bounded-index subgroup of $G$ which is $l$-Engel for some bounded $l$.
Since $[A, g] / B_g$ is bounded, there is a bounded number $f$ such that $[A, {}_fg] \leq B_g$ for all $g \in G$.
For $g \in G$ and $i \geq 0$ write $B_{i,g}$ for $[B_g, {}_ig]$, and note that $B_{i,g}$ is normal in $G$.
Let $T_{i,g} = C_G(B_{i,g})$.
Then $B_g = B_{0,g} \geq B_{1,g} \geq \cdots$
and $T_g = T_{0,g} \leq T_{1,g} \leq \cdots$.
Let $\beta_i = \max_{g \in G} [G : T_{i,g}]$, so $\beta_0 \geq \beta_1 \geq \cdots$.
Since $\beta_0 = \max_{g \in G} [G : T_g]$ is bounded, there is a bounded number $i$ such that $\beta_i = \beta_{2i+f}$.
Pick $g$ so that $[G : T_{i,g}] = \beta_i = \beta_{2i+f}$.
Then for any $h \in T_{i,g}$ we have
\[
    B_{2i+f, g} = [B_{i,g}, {}_{f+i} g] = [B_{i,g}, {}_{f+i} gh] \leq B_{i,gh},
\]
so $T_{i,gh} \leq T_{2i+f,g} = T_{i,g}$.
Since $[G:T_{i,g}] = \beta_i$, this implies that $T_{i,gh} = T_{i,g}$ whenever $h \in T_{i,g}$.
Hence $T_{i,g}$ centralizes $B_{i,gh}$ for all $h \in T_{i,g}$.

Let $H = T_{i,g}$ and $Z = Z(H)$ and $\bar G = G/Z$.
Let $h \in H$.
Since $Z$ contains $B_{i,gh}$, it follows that $\bar{gh}$ is $(f + i + 1)$-Engel in $\bar G$.
In particular $\bar g$ is $(f+i+1)$-Engel.
Applying \Cref{lll}, $\bar h$ is $(2f+2i+3)$-Engel.
Hence $\bar H$ is $(2f+2i+3)$-Engel.
Applying \Cref{engel}, $H$ is nilpotent of bounded class,
as required.

\subsection{The coprime case}

\begin{lemma}
    Let $A$ be a finite nilpotent group
    and let $\Gamma$ be a group acting on $A$ with kernel $K$.
    Assume $|A|$ and $|\Gamma / K|$ are coprime.
    Then $[A, \Gamma] = [A, \Gamma, \Gamma]$.
\end{lemma}
\begin{proof}
    We may assume $K = 1$ and $|A|$ and $|\Gamma|$ are coprime.
    Since $A$ is the direct product of its Sylow subgroups, which are preserved by $\Gamma$,
    we may assume $A$ is a $p$-group for some prime $p$,
    and so $\Gamma$ is a $p'$-group.
    Now see \cite{aschbacher}*{(24.5)}.
\end{proof}

We now consider the case of \Cref{prop:bounded-class} in which $G$ is finite and metabelian,
$G / G'$ is a $p$-group, and $G'$ is a $p'$-group.
As before let $A = G'$ and let $B_g, T_g$ be normal subgroups of $G$ such that $B_g \leq [A, g]$,
the indices $[[A, g] : B_g]$ and $[G : T_g]$ and the order of $[B_g, T_g]$ are bounded.
Then $C_G([B_g, T_g])$ has bounded index in $G$,
so we may additionally assume that $T_g$ centralizes $[B_g, T_g]$ by replacing $T_g$ with $C_{T_g}([B_g, T_g])$ if necessary.
Similarly we may assume $T_g$ centralizes $[A, g] / B_g$.
But then two applications of the lemma show
\[
    [A, g, T_g] = [A, g, T_g, T_g] \leq [B_g, T_g] = [B_g, T_g, T_g] = 1.
\]
Hence we may assume $B_g = [A, g]$ and $T_g = C_G([A, g])$.

Choose $g \in G$ so that $[G : T_g]$ is maximal.
Let $h \in T_g$.
By the lemma,
\[
    [A, g] = [A, g, g] = [A, g, gh] \leq [A, gh].
\]
It follows that $T_{gh} \leq T_g$.
Since $[G : T_g]$ is maximal, $T_{gh} = T_g$.
Hence $T_g$ centralizes both $[A, g]$ and $[A, gh]$, and $[A, h] \leq [A, g][A, gh]$, so $T_g$ centralizes $[A, h]$.
Since this holds for every $h \in T_g$, $[A, T_g, T_g] = 1$.
Since $T_g' \leq A$, it follows that $T_g$ is nilpotent (and in fact abelian, since $G$ has no nonabelian nilpotent subgroups).

\subsection{Finite metabelian groups in general}

Let $G$ be a finite metabelian group satisfying \eqref{eq:covering-condition}.
Let $e = n!$. Then $G^e$ centralizes $B$.
The group $G^e/Z(G^e)$ has at most $|S|$ commutators so it has a subgroup of bounded index of nilpotency class at most 2,
so $G^e$ has a bounded-index subgroup of class at most 3.

Since $G / G'$ is finite and abelian it is the direct product of its Sylow $p$-subgroups.
For each prime $p \leq n$ let $G_p$ be the preimage in $G$ of the Sylow $p$-subgroup of $G / G'$.
For $p > n$, the Sylow $p$-subgroups of $G$ are contained in $G^e$.
Hence $G = G^e \prod_{p \leq n} G_p$,
so by Fitting's theorem it suffices to prove the result for $G = G_p$.
Thus we may assume $G / G'$ is a $p$-group.

Let $P$ be a Sylow $p$-subgroup of $G$.
Then $G = G' P$.
By the subsection on $p$-groups, we can replace $P$ with a subgroup of bounded index and bounded class,
so without loss of generality $P$ has bounded class $c$ say.
Since the action of $G$ on $G'$ factors through $P$, it follows that $[P \cap G', {}_c G] = 1$,
so $P \cap G' \leq Z_c(G)$.
Factoring out $Z_c(G)$, we may thus assume $G'$ is a $p'$-group.
By the result of the coprime section we are done.

\section{Bounded-class groups}
\label{sec:to-class-4}

We need the following result from \cite{FM},
which is a local version of a classical result of Baer.

\begin{theorem}[\cite{FM}*{Theorem~B}]
    \label{thm:local-baer}
    Let $G$ be a group such that $n = [\gamma_s(G) : Z_t(G) \cap \gamma_s(G)]$ is finite.
    Then $[\gamma_{s+1}(G) : Z_{t-1}(G) \cap \gamma_{s+1}(G)]$ is finite and $(n,s)$-bounded.
\end{theorem}

Now we will complete the proofs of \Cref{thm:d2-theorem,thm:covered-groups}.
By \Cref{cor:d2-neumann}, it suffices to prove \Cref{thm:covered-groups}.
By \Cref{prop:bounded-derived-length,prop:bounded-class}, $G$ has a subgroup of bounded index and bounded nilpotency class.
Replacing $G$ with this subgroup, we may assume $G$ has bounded class, say $c$.
In particular $G$ is locally residually finite, so by \Cref{lem:reduction-to-finite,lem:covering-subgroups-and-quotients} we may assume $G$ is finite.
We will show by induction on $c$ that $G$ has a subgroup $H$ such that $[G : H]$ and $|\gamma_4(H)|$ are both $(n, |S|, c)$-bounded ($G$ is bounded-by-class-3-by-bounded);
if we can show this then, since $C_H(\gamma_4(H))$ has class 4, we are done.
By induction $G / Z(G)$ has a subgroup $H / Z(G)$ such that $[G : H]$ and $[\gamma_4(H) : \gamma_4(H) \cap Z(G)]$ are both bounded.
Applying \Cref{thm:local-baer}, $|\gamma_5(H)|$ is bounded.
Hence, passing to the section $H / \gamma_5(H)$, we may assume $G$ has class 4.
Moreover, by the converse of \Cref{thm:metric-neumann}, $d_2(G)$ is bounded away from zero.
Hence the proof is completed by the following proposition.

\begin{proposition}
%    \label{prop:class-4}
    Let $G$ be a finite group such that $\gamma_5(G) = 1$ and $d_2(G) \geq \eps > 0$.
    Then $G$ has a subgroup $H$ such that $[G : H]$ and $|\gamma_4(H)|$ are both $\eps$-bounded.
\end{proposition}
\begin{proof}
    In this proof, ``bounded'' means ``$\eps$-bounded''.
    Let $\gg = \gg_1 \oplus \gg_2 \oplus \gg_3 \oplus \gg_4$ be the graded Lie ring associated to $G$,
    where $\gg_i = \gamma_i(G) / \gamma_{i+1}(G)$ for $1 \leq i \leq 4$.
    Since $d_2(G) \geq \eps$,
    \begin{equation}
        \label{eq:g1^3}
        \P_{x, y, z \in \gg_1}([x, y, z] = 0) \geq \eps.
    \end{equation}
    We claim that also
    \begin{equation}
        \label{eq:g1^2xg2}
        \P_{x, y \in \gg_1, z \in \gg_2}([x, y, z] = 0) \geq \eps.
    \end{equation}
    Since
    \[
        \P_{x, y, z_0 \in G, z \in G'}([x, y, zz_0] = 1) = \P_{x, y, z \in G} ([x, y, z] = 1) \geq \eps,
    \]
    there is some $z_0 \in G$ such that
    \[
        \P_{x, y \in G, z \in G'} ([x, y, zz_0] = 1) \geq \eps.
    \]
    Pick $z_1 \in G'$ such that $[x, y, z_1z_0] = 1$.
    Now for $z \in G'$ we have
    \[
        [x, y, zz_0] = [x, y, z] [x, y, z_0],
    \]
    and for fixed $x$ and $y$ the map $z \mapsto [x, y, z]$ is a homomorphism for $z \in G'$,
    so
    \[
        [x, y, zz_0] = 1 \implies [x, y, zz_1^{-1}] = 1.
    \]
    It follows that
    \[
        \P_{z \in G'}([x, y, zz_0] = 1) \leq \P_{z \in G'}([x, y, z] = 1).
    \]
    Hence
    \[
        \P_{x, y \in G, z \in G'}([x, y, z] = 1) \geq \P_{x, y \in G, z \in G'} ([x,y,zz_0] = 1) \ge \eps.
    \]
    Descending to $\gg$ gives \eqref{eq:g1^2xg2}.
    
    From \eqref{eq:g1^3} and \eqref{eq:g1^2xg2} and \Cref{thm:bias-structure}, we have have expressions
    \begin{align*}
        [x, y, z] &= F_1(x, y, z) + \cdots + F_7(x, y, z) && (x, y, z \in \gg_1)\\
        [x, y, z] &= F'_1(x, y, z) + \cdots + F'_7(x, y, z) && (x, y \in \gg_1, z \in \gg_2),
    \end{align*}
    where
    \begin{equation}
        \label{eq:F1-F7}
        \begin{aligned}
        F_1(x, y, z) &= G_1(g_1(x), y, z),\\
        F_2(x, y, z) &= G_2(g_2(y), x, z),\\
        F_3(x, y, z) &= G_3(g_3(z), x, y),\\
        F_4(x, y, z) &= G_4(g_4(x, y), z),\\
        F_5(x, y, z) &= G_5(g_5(x, z), y),\\
        F_6(x, y, z) &= G_6(g_6(y, z), x),\\
        F_7(x, y, z) &= G_7(g_7(x, y, z)),
        \end{aligned}
    \end{equation}
    likewise for $F'_1, \dots, F'_7$,
    and where $|\cod(g_i)|$ and $|\cod(g'_i)|$ are $\eps$-bounded for each $i \in \{1, \dots, 7\}$.
    Hence we obtain two expressions for $[x, y, z, w]$:
    \begin{equation}
        \label{eq:expr1}
        [x, y, z, w] = \sum_{i=1}^7 [F_i(x, y, z), w]
    \end{equation}
    and, by the Jacobi identity,
    \begin{align}
        [x, y, z, w] 
       &= [x, y, w, z] + [x, y, [z, w]] \nonumber\\
       &= \sum_{i=1}^7 [F_i(x, y, w), z] + \sum_{i=1}^7 F_i'(x, y, [z, w]). \label{eq:expr2}
    \end{align}
    Note that \eqref{eq:expr1} only has terms of the forms
    \begin{align*}
        &G(g(x), y, z, w), G(x, g(y), z, w), G(x, y, g(z), w), \\
        &G(g(x, y), z, w), G(g(x, z), y, w), G(g(y, z), x, w), \\
        &G(g(x, y, z), w);
    \end{align*}
    while \eqref{eq:expr2} only has terms of the forms
    \begin{align*}
        &G(g(x), y, z, w), G(x, g(y), z, w), G(x, y, z, g(w)), \\
        &G(g(x, y), z, w), G(g(x, w), y, z), G(g(y, w), x, z), G(x, y, g(z, w)), \\
        &G(g(x, y, w), z), G(g(x, z, w), y), G(g(y, z, w), x), \\
        &G(g(x, y, z, w)).
    \end{align*}
    Applying \Cref{lem:rank-refinement}, there is a bounded-index subgroup $\hh_1 \leq \gg_1$ such that for $(x, y, z, w) \in \hh_1^4$
    we have an expression for $[x, y, z, w]$ of the form
    \[
        [x,y,z,w] = F_1(f_1(x, y), z, w) + F_2(f_2(x, z), f'_2(y, w)) + F_3(f_3(x, w), f'_3(y, z)),
    \]
    where $f_1, f_2, f_2', f_3, f_3'$ have bounded codomains.
    The latter two terms take values in a bounded part of $\gg_4$, which we may quotient out, so without loss of generality
    \[
        [x, y, z, w] = F_1(f_1(x, y), z, w).
    \]
    Now if $f_1(x, y) = f_1(x', y')$ we have $[x, y] - [x', y'] \in Z_2(\hh)$,
    where $\hh$ is the subring generated by $\hh_1$.
    
    Let $H$ be the preimage of $\hh_1$ in $G$.
    Then $[G:H] = [\gg_1:\hh_1]$.
    For $x \in H$ write $\bar x$ for the image in $\hh_1$.
    Then for $x, y, x', y' \in H$ if $f_1(\bar x, \bar y) = f_1(\bar{x'}, \bar{y'})$ then $[x, y] [x', y']^{-1} \in Z_2(H)$.
    Hence there are only boundedly many commutators in $\gamma_2(H) / Z_2(H)$,
    so $[\gamma_2(H) : Z_2(H)]$ is bounded by Neumann's BFC theorem (\Cref{thm:BFC}).
    Applying \Cref{thm:local-baer} twice, $|\gamma_4(H)|$ is bounded, as claimed.
\end{proof}

\begin{remark}
    Let $G$ be a finite group with $d_2(G) \geq \eps > 0$ as in the statement of \Cref{thm:d2-theorem}.
    We have established the existence of a subgroup $H$ of class 4 such that $[G:H]$ and $|\gamma_4(H)|$ are bounded.
    The covering condition and \Cref{lem:covering-subgroups-and-quotients} (or \eqref{eq:dk-submult}) shows that $d_2(H / \gamma_4(H))$ is bounded away from zero,
    so by replacing $G$ with $H / \gamma_4(H)$ we may assume $G$ has class 3.
    Now consider the expression
    \[
        [x,y,z] = F_1(x, y, z) + \cdots + F_7(x, y, z) \qquad (x, y, z \in G / G')
    \]
    established in the proof above, where $F_1, \dots, F_7$ are given by \eqref{eq:F1-F7}.
    By passing to the bounded-index subgroup $\ker g_1 \cap \ker g_2 \cap \ker g_3$ we can kill off $F_1, F_2, F_3$.
    The image of $F_7$ is contained in a bounded-size subgroup of $\gamma_3(G)$, so we may quotient it out.
    Thus we are left with
    \[
        [x, y, z] = G_4(g_4(x, y), z) + G_5(g_5(x, z), y) + G_6(g_6(y, z), x).
    \]
    Conversely, suppose $G$ has class 3 and the triple commutator has an expression of this form.
    Then for any $x \in G$, the proportion of $y \in G$ such that $g_4(x, y) = 0$ is at least $|\cod(g_4)|^{-1}$,
    and for any $x, y \in G$ the proportion of $z \in G$ such that $g_5(x, z) = 0$ and $g_6(y, z) = 0$ is at least $|\cod(g_5)|^{-1} |\cod(g_6)|^{-1}$.
    It follows that $d_2(G) \geq |\cod(g_4)|^{-1} |\cod(g_5)|^{-1} |\cod(g_6)|^{-1}$.
\end{remark}
    
\bibliography{refs}
\bibliographystyle{plain}

\end{document}